\documentclass{amsart}

\usepackage{amsmath,amssymb,amsthm}

\newtheorem{thm}{Theorem}
\newtheorem{prop}[thm]{Proposition}

\newtheorem{lem}[thm]{Lemma}

\newtheorem{defn}[thm]{Definition}

\newcommand{\Z}{\mathbf{Z}}
\newcommand{\M}{\mathsf M}
\newcommand{\hh}{\mathsf h}
\newcommand{\HH}{\mathsf H}
\newcommand{\dd}{\mathsf d}
\newcommand{\DD}{\mathsf D}
\newcommand{\Mi}{\mathsf M_{\text{\rm int}}}
\newcommand{\eps}{\varepsilon}

\title{Faster deterministic integer factorization}
\author{Edgar Costa}
\address[Edgar Costa]{Courant Institute of Mathematical Sciences \\
New York University \\
251 Mercer Street \\
New York, N.Y. 10012-1185 \\
U.S.A}
\thanks{The first author was partially supported by FCT doctoral grant SFRH/BD/69914/2010.}
\email{edgarcosta@nyu.edu}
\author{David Harvey}
\address[David Harvey]{School of Mathematics and Statistics \\
University of New South Wales \\
Sydney NSW 2052 \\
Australia}
\thanks{The second author was partially supported by the Australian Research Council, DECRA Grant DE120101293.}
\email{d.harvey@unsw.edu.au}

\begin{document}

\maketitle

\begin{abstract}
The best known unconditional deterministic complexity bound for computing the prime factorization of an integer $N$ is $O(\Mi(N^{1/4} \log N))$, where $\Mi(k)$ denotes the cost of multiplying $k$-bit integers. This result is due to Bostan--Gaudry--Schost, following the Pollard--Strassen approach. We show that this bound can be improved by a factor of $\sqrt{\log \log N}$.
\end{abstract}

\section{Introduction}

In this paper we consider unconditional deterministic complexity bounds for computing the prime factorization of a positive integer $N$. Complexity refers to bit complexity, in the sense of the multitape Turing machine model \cite{Pap-complexity}.

The best known bounds for this problem are all of the shape $O(N^{1/4+\eps})$. The simplest algorithm achieving such a bound is due to Strassen \cite{Str-factoring}. Its complexity is analyzed in \cite{BGS} and shown to be
 \[ O(\Mi(N^{1/4} \log N) \log N), \]
where $\Mi(k)$ denotes the cost of multiplying $k$-bit integers. (The best known bound for $\Mi(k)$ is $\Mi(k) = O(k \log k \, 2^{\log^* k})$ where $\log^* k$ denotes the iterated logarithm \cite{Fur-faster}.) Bostan, Gaudry and Schost \cite{BGS} improved this further to
 \[ O(\Mi(N^{1/4} \log N)). \]
Our main result is the following refinement.
\begin{thm}\label{thm:main}
There exists a deterministic algorithm that computes the prime factorization of a positive integer $N$ in
 \[ O\left(\Mi\left(\frac{N^{1/4} \log N}{\sqrt{\log \log N}}\right)\right) \]
bit operations.
\end{thm}

To explain the main idea of our algorithm, we recall Strassen's approach. Consider the simplest situation where $N$ is a product of two distinct primes, say $N = pq$, $p < q$. Let $K = \lfloor N^{1/2}\rfloor$. Since $p \leq K$ and $q > K$ we have $\gcd(K! \bmod N, N) = p$, so it suffices to compute $K! \bmod N$. To simplify further, assume that $K = L^2$ for some integer $L$. Strassen observes that
 \[ K! = f(0) f(L) f(2L) \cdots f((L-1) L) \]
where
 \[ f(x) = (x+1)(x+2) \cdots (x+L). \]
He computes $f(x)$ in $(\Z/N\Z)[x]$ using a product tree, and then evaluates $f(x)$ at $0, L, \ldots, (L-1) L$ using a fast multipoint evaluation algorithm. The overall complexity is quasilinear in $L = O(N^{1/4})$. The algorithm of \cite{BGS} evaluates the same product, but uses a more involved evaluation scheme that saves a factor of $O(\log L)$.

Our key observation is that $K!$ has many terms that do not contribute any useful information. For example, it is easy to extract factors of $2$ from $N$. Once this is done, we may assume $N$ is odd, so any remaining factors must be odd. Thus we should replace $K!$ by a product of the form $1 \times 3 \times 5 \times \cdots \times K'$. This immediately saves a factor of $\sqrt 2$ in Strassen's algorithm (or in the Bostan--Gaudry--Schost algorithm).

More generally, we may select a bound $B$ and remove from $N$ all prime factors bounded by $B$, and then replace the factorial by a generalized factorial that omits all integers divisible by any of these primes. This is a similar idea to the `factorial sieving' performed in \cite{crandall}. Our contribution is to show that the algorithm of \cite{BGS} can be modified to handle such generalized factorials. Choosing a larger $B$ leads to greater savings, but also imposes a cost due to the more complex pattern of integers removed from the generalized factorial. Optimizing the choice of $B$ leads to the bound given in Theorem \ref{thm:main}.

All of the factorization algorithms mentioned above (including ours) are of theoretical interest only, and none of them are remotely practical. If we allow probabilistic algorithms, or complexity arguments that depend on unproved hypotheses such as the Riemann Hypothesis, then much better bounds can be achieved. For this we refer the reader to the excellent survey \cite{CP-primes}.

\section{Fast polynomial evaluation on arithmetic progressions}

In this section $R$ denotes a ring, in which we can multiply and sum elements in $m$ bit operations, and for which polynomials in $R[x]$ of degree $d$ can be multiplied in $\M(d)$ bit operations. We will only provide high-level descriptions of all algorithms and skip the details of their corresponding Turing machine implementations. We assume that $\M(d)$ behaves reasonably, in particular that $\M(d d') \geq d \M(d')$, and so $\M(d) \geq dm$. In the next section we will specialize to the case $R = \Z/N\Z$.

We will often use the following standard result without comment. For a proof see \cite[Lemma 1]{BGS}.
\begin{lem}
Suppose that $r_1,\dots,r_d$ are invertible in $R$. Given
 $r_1,\dots,r_d$ and $(r_1\dots r_d)^{-1}$, we may compute $r_1 ^{-1},
 \dots, r_d ^{-1}$ in $O(d m )$ bit operations.
\end{lem}

Our basic tool will be \cite[Theorem 5]{BGS}, which is given as Proposition \ref{prop:shifting} below. To state it, we introduce the following notation.
\begin{defn}
Let $\alpha, \beta \in R$ and $d \geq 1$. We say that $\hh(\alpha, \beta, d)$ is satisfied if the elements
$$\beta,\quad 2,\dots,d,\quad (\alpha-d \beta),(\alpha-(d-1) \beta),\dots,(\alpha+d \beta)$$
are invertible in $R$, and we put
$$\dd(\alpha, \beta, d) = \beta 2 \cdots d (\alpha-d \beta)( \alpha-(d-1) \beta)\cdots (\alpha+d \beta).$$
\end{defn}
Thus $\hh(\alpha,\beta,d)$ holds if and only if $\dd(\alpha,\beta,d)$ is invertible.

\begin{prop}\label{prop:shifting}
Let $\alpha,\beta \in R$ and $d \geq 1$. Assume that $\hh(\alpha,\beta,d)$ holds,
and that the inverse of $\dd(\alpha,\beta,d)$ is known. Let $F$ be a polynomial in $R[x]$ of degree at most $d$.
Given
$$ F(0),F(\beta),\dots,F(d\beta),$$
we may compute
$$F(\alpha),F(\alpha+\beta),\dots,F(\alpha+d\beta)$$
in $O( \M(d) )$ bit operations.
\end{prop}
\begin{proof}
See \cite[Theorem 5]{BGS}; the proof is based on the Lagrange interpolation formula. We emphasize that the coefficients of $F(x)$ are \emph{not} part of the input.
\end{proof}

Let $H \in R[x]$ be a polynomial of degree $\rho \geq 1$. In Section \ref{application} we will be interested in evaluating the polynomial
 \[ H_k(x) = H(x) H(x+1) \cdots H(x + k - 1) \]
on a certain arithmetic progression. Theorem 8 of \cite{BGS} gives an efficient solution to this problem for $\rho = 1$. The following two results generalize this to the case $\rho \geq 1$.
\begin{prop}
\label{prop:extend}
Let $\beta \in R$ and $k \geq 1$. Assume that  $$ \hh(k, \beta, k \rho ) \quad \text{\rm and} \quad \hh(( k \rho+1)\beta, \beta,k \rho )$$
both hold and that the inverses of 
$$ \dd(k, \beta, k \rho) \quad \text{\rm and} \quad \dd(( k \rho+1)\beta, \beta,k \rho)$$
are known. Given
 $$ H_k(0), H_k(\beta), \ldots, H_k(k\rho\beta), $$
we may compute
 $$ H_{2k}(0), H_{2k}(\beta), \ldots, H_{2k}(2k\rho\beta) $$
in $O(\M(k \rho))$ bit operations.
\end{prop}
\begin{proof}
We start by applying Proposition \ref{prop:shifting} with $\alpha = k$ and $d = k\rho$ to the known values of $H_k(x)$ to obtain
$$ H_k(k), H_k(\beta + k),\ldots, H_k(k \rho \beta + k)$$
in $O(\M(k \rho))$ bit operations. Since
$$ H_{2k}(x) = H_k(x) H_k(x+ k) $$
we may multiply these to obtain
 $$ H_{2k}(0), H_{2k}(\beta), \ldots, H_{2k}(k\rho\beta) $$
in $(k \rho + 1) m = O(\M(k \rho))$ bit operations.

We now apply Proposition \ref{prop:shifting} to the original values again, this time with $\alpha = (k \rho + 1)\beta$, to obtain
$$ H_k((k \rho+1)\beta), H_k((k \rho+2)\beta), \ldots, H_k((2k\rho+1)\beta).$$
A final application of Proposition \ref{prop:shifting} with $\alpha = k$ yields
$$ H_k((k \rho+1)\beta + k), H_k((k \rho+2)\beta + k), \ldots, H_k((2k\rho+1)\beta + k).$$
As above we can multiply these to obtain
$$ H_{2k}((k \rho+1)\beta), H_{2k}((k \rho+2)\beta), \ldots, H_{2k}((2k\rho+1)\beta).$$
Discarding the last value, we have the desired output. The total complexity is $O(\M(k \rho))$ bit operations.
\end{proof}

In Proposition \ref{prop:fasteval}, we will apply the previous result recursively. The following definition consolidates the required invertibility conditions.
\begin{defn}
Let $r \geq 1$. We say that $\HH(2^r,\beta,\rho)$ holds if $\hh(2^i,\beta, 2^i \rho)$ and $\hh(( 2^i \rho +1)\beta, \beta , 2^i \rho)$ hold for each $0 \leq i < r$. We write
 \[ \DD(2^r, \beta, \rho ) = \prod_{i=0} ^{r-1} \dd(2^i ,\beta , 2^i \rho ) \dd((
2^i \rho +1)\beta, \beta , 2^i \rho ). \]
\end{defn}
As before, $\HH(2^r,\beta,\rho)$ holds if and only if $\DD(2^r,\beta,\rho )$ is invertible. 

\begin{prop}\label{prop:fasteval}
Assume that $\HH(2^r,\beta,\rho )$ holds and that the inverse of $\DD(2^r,\beta,\rho )$ is known. Let $k=2^r$. We may compute $$H_k(0), H_k(\beta), \ldots, H_k(k\rho\beta) $$ in $O(\M(k \rho )+ \rho^2 m)$ bit operations.
\end{prop}

\begin{proof}
We first compute $H(x)$ at $x = 0, \beta, \ldots, \rho \beta$. This can be done in $O(\rho^2 m)$ bit operations. (This can be improved to $O( \M(\rho) \log \rho)$ using standard multipoint evaluation techniques, but we will not use this.)

We then apply Proposition \ref{prop:extend} successively for $k = 1, 2, 4, \ldots, 2^{r-1}$. The cost at the $i$th step is $O(\M(2^i \rho))$ bit operations. At the $i$th step, we need to supply the inverses of
$$ \dd(2^i, \beta, 2^i \rho ) \quad \text{and} \quad \dd((2^i \rho +1)\beta, \beta,2^i\rho ). $$
Computing each product can be done in $O(2^i \rho m)$ bit operations, and with the inverse of $\DD(2^r,\beta,\rho)$ we can compute the inverses sought; all
this can be done in $O(k \rho m)$ bit operations. The total complexity is
 \[ O(\rho^2 m + k \rho m + \M(\rho) + \M(2\rho) + \cdots + \M(2^{r-1} \rho)) = O(\rho^2 m + \M (k \rho)) \]
bit operations.
\end{proof}

\section{Application to integer factorization}
\label{application}

We now specialize to $R = \Z/N\Z$. Elements of $R$ are represented in the standard way using bitstrings of length $O(\log N)$. We have $m = O(\Mi(\log N))$, and $\M(d) = O(\Mi(d \log (dN)))$ using Kronecker substitution \cite{Sch-numerical}. If $d = O(N)$, which for us will always be the case, this simplifies to $\M(d) = O(\Mi(d \log N))$. The inverse of an element of $R$, if it exists, may be computed in time $O(\Mi(\log N) \log \log N)$ using a fast extended GCD algorithm \cite{gcd}.

Let $B>2$ be a parameter; an optimal value for $B$ will be chosen later on. Let
 \[ Q = \prod_{\substack{p < B \\ \text{$p$ prime}}} p. \]
We will apply the results of the previous section to the polynomial
 \[ H(x) = \prod_{\substack{j=1\\(j,Q)=1}} ^{Q} (Qx+j), \]
which has degree $\rho = \phi(Q) = \prod_{p < B} (p-1)$.

We start with an auxiliary result.
\begin{lem}
\label{lem:invertible}
Let $f_0,\dots,f_{k-1} \in \Z/N\Z$. Then we can decide if all $f_i$ are invertible modulo $N$ and, if not, find a noninvertible $f_i$ in
$$O(k \, \Mi(\log N)+\log k \; \Mi (\log N) \log \log N)$$
bit operations.
\end{lem}
\begin{proof}
See \cite[Lemma 12]{BGS}. The idea is to apply the GCD to the subproduct tree formed by the $f_i$.
\end{proof}

The core of our algorithm comes next.

\begin{prop}\label{prop:main}
Let $r \geq 0$ and $b = 4^r \rho Q$. Assume that $b < N$ and that $(N, Q) = 1$. We can find a prime divisor $\ell$ of $N$ such that $\ell \leq b$, or prove that no such divisor exists, in
\begin{equation*}
O\left(\Mi\left( 2^r \rho \log N \right) + (Q^2 + \log (2^r \rho)) \Mi (\log N) \log \log N \right)
\end{equation*}
bit operations. 
\end{prop}
\begin{proof}
We first list the integers $1 \leq j < Q$ such that $(j, Q) = 1$, by computing $(j, Q)$ for each candidate $j$. Noting that $Q < N$, this uses $O(Q \, \Mi(\log N) \log \log N)$ bit operations. Using this list, we compute the coefficients of $H(x)$; the naive algorithm for this uses $O(\rho^2 \Mi(\log N))$ bit operations.

In the algorithm described below, we will test various elements of $\Z/N\Z$ for invertibility. If at any stage we encounter a noninvertible $x$ with $x \leq b$, then we are done. Indeed, to find a suitable prime divisor of $N$ it suffices to perform trial division of $x$ by the integers $2 \leq \ell \leq \sqrt b$ with $(\ell, Q) = 1$. The number of such $\ell$ is at most $\rho\lceil \sqrt b /Q \rceil \leq \rho(\sqrt b/Q + 1) = 2^r \rho \sqrt{\rho/Q} + \rho = O(2^r \rho)$, so the cost of these trial divisions is $O(2^r \rho \Mi(\log N)) = O(\Mi(2^r \rho \log N))$.

We would like to apply Proposition \ref{prop:fasteval} to $H(x)$ with $k = \beta = 2^r$. We must first verify that $\HH(2^r, 2^r, \rho)$ is satisfied. This is equivalent to invertibility of
 \[  2, 3, \ldots, (2^r \rho + 1) \]
and
 \[ (2^i - 2^i \rho 2^r),(2^i -(2^i \rho-1) 2^r),\ldots,(2^i + 2^i \rho 2^r)  \]
for each $0 \leq i \leq r-1$. These integers are all bounded (in absolute value) by $b$, and there are $O(2^r \rho)$ of them. By Lemma \ref{lem:invertible} we may prove they are invertible, or find a noninvertible one, in
 \[ O(2^r \rho \Mi(\log N) + \log (2^r \rho) \Mi(\log N) \log \log N) \]
bit operations. Computing $\DD(2^r, 2^r, \rho)$ requires $O(2^r \rho \Mi(\log N))$ bit operations, and finding its inverse has negligible cost. Proposition \ref{prop:fasteval} then computes
 \[ H_k(0), H_k(k), \ldots, H_k((k \rho - 1)k) \]
using
 \[ O(\Mi(2^r \rho \log N) + \rho^2 \Mi(\log N)) \]
bit operations.

By construction we have
\begin{equation}
\label{eq:product}
\prod_{i=0} ^{k \rho -1} H_k(i k) = \prod_{i=0} ^{k \rho -1} \prod_{ \substack{j=1 \\ (j,Q)=1}} ^{kQ} (i k Q +j)  = \prod_{\substack{j=1\\ (j,Q)=1 }} ^{ b} j.
\end{equation}
If any of the $H_k(ik)$ are noninvertible, by Lemma \ref{lem:invertible} we may find one in
 \[ O(2^r \rho \Mi(\log N) + \log(2^r \rho) \Mi(\log N) \log \log N) \]
bit operations. In this case we may find a noninvertible integer bounded by $b$ within the same time bound, since $H_k(ik)$ is itself a product of $O(k \rho)$ integers bounded by $b$. Otherwise we have proved that \eqref{eq:product} is invertible, and we are finished.
\end{proof}

\begin{proof}[Proof of Theorem \ref{thm:main}]
We will take $B = \frac{1}{11} \log N$. By the prime number theorem we have $\sum_{p < x} \log p = x + o(x)$ \cite[Theorem 6.9]{MNT1}, so
 \[ Q = \prod_{p < B} p = O(N^{(1+o(1))/11}) = O(N^{1/10}). \]

We may remove any factors of $N$ bounded by $B$ with negligble cost, so we may assume that $(N, Q) = 1$.

We now apply Proposition \ref{prop:main}, starting with $b = \rho Q$ ($r = 0$). If we find a prime divisor $\ell \leq b$, we remove it from $N$ and repeat. Otherwise we quadruple $b$ (increment $r$) and repeat. We continue until we reach $b \geq \sqrt{N}$; the last iteration has $r = r_0$ where
 \[ r_0 = \lceil \log_4(\sqrt N/\rho Q) \rceil. \]

To analyze the overall complexity, observe that when we apply the algorithm for a given $b$, all prime divisors $\ell \leq b/4$ have already been found and removed. Since their product is bounded by $N$, the number of runs of the algorithm for a given $b$ is bounded by $O(\log N / \log b)$. Therefore the complexity is
 \[ O\left(\sum_{r=0}^{r_0} \frac{\log N}{\log(4^r \rho Q)} \left(\Mi\left( 2^r \rho \log N \right) + (Q^2 + \log (2^r \rho)) \Mi (\log N) \log \log N \right) \right). \]

Since $Q^2 = O(N^{1/5})$ and $r_0 = O(\log N)$, the second term is bounded by $O(N^{1/5} \log^{3+\eps} N)$. To estimate the first term, we split the sum into $r \leq r_0/2$ and $r > r_0/2$. For the terms with $r \leq r_0/2$, we have $2^r = O(N^{1/8}/ (\rho Q)^{1/4}) = O(N^{1/8} / \rho^{1/2})$ so $2^r \rho = O(N^{1/8} \rho^{1/2}) = O(N^{1/8+1/20}) = O(N^{1/5})$; thus the sum is bounded by $(\log N)^2 (N^{1/5} \log N)^{1+\eps}$. So far these contributions are negligible. The main contribution comes from the terms $r > r_0/2$. For these $r$ we have $4^r \rho Q \geq N^{1/4}$, so $\log N / \log(4^r \rho Q) = O(1)$, and the sum is bounded by
\begin{align*}
 \sum_{r_0/2 < r \leq r_0} O(\Mi(2^r \rho \log N)) & = O(\Mi(2^{r_0} \rho \log N)) \\ & = O(\Mi(N^{1/4} (\rho/Q)^{1/2} \log N)).
\end{align*}
But by Mertens' theorem \cite[Theorem 2.7]{MNT1},
 \[ \frac{\rho}{Q} = \prod_{p < B} \frac{p-1}p = O\left(\frac{1}{\log B}\right) = O\left(\frac{1}{\log \log N}\right), \]
and the desired result follows.
\end{proof}

\bibliographystyle{alpha}
\bibliography{biblio}

\end{document}